\documentclass[a4paper]{article}
\pdfoutput=1
\usepackage{amsmath,amssymb,amsthm}
\usepackage[english]{babel}
\usepackage{hyperref}

\theoremstyle{plain}
\newtheorem{theorem}{Theorem}
\newtheorem{proposition}[theorem]{Proposition}
\newtheorem{corollary}[theorem]{Corollary}

\theoremstyle{definition}
\newtheorem{definition}[theorem]{Definition}
\newtheorem{example}[theorem]{Example}

\newcommand{\wt}[1]{\mathrm{wt}(#1)}
\newcommand{\supp}[1]{\mathrm{supp}(#1)}
\newcommand{\GF}[1]{\mathrm{GF}(#1)}
\newcommand{\PG}[1]{\mathrm{PG}(#1)}
\newcommand{\AG}[1]{\mathrm{AG}(#1)}

\title{Weight enumeration of codes from finite spaces}
\author{Relinde Jurrius}

\begin{document}

\maketitle

\begin{abstract}
We study the generalized and extended weight enumerator of the $q$-ary Simplex code and the $q$-ary first order Reed-Muller code. For our calculations we use that these codes correspond to a projective system containing all the points in a finite projective or affine space. As a result from the geometric method we use for the weight enumeration, we also completely determine the set of supports of subcodes and words in an extension code. 
\end{abstract}

\section{Introduction}

The weight enumerator and the set of supports are important invariants of a linear code. Besides their intrinsic importance as mathematical objects, they are used in the probability theory involved with different ways of decoding. In this paper we will determine the generalized and extended weight enumerator, as well as the set of supports, of codes with a geometric structure. \\

The geometrical method for determining the generalized weight enumerator is described by Tsfasman and Vl\v{a}dut in \cite{tsfasman:1995}. Several examples are worked out in the paper. We will apply their methods to another class of codes, namely codes that correspond to a projective system containing all the points in a finite projective or affine space. This are the $q$-ary Simplex code and the $q$-ary first order Reed-Muller code. \\
After finding explicit formulas for the generalized weight enumerator of these codes, we will apply the correspondence with the extended weight enumerator in \cite{jurrius:2011} to find explicit formulas for the extended weigh enumerator of codes coming from finite projective and affine spaces. This formulas were found before by Mphako \cite{mphako:2000}, using the theory of matroids and geometric lattices. \\

As a result from the method used in this paper, we will not only find information about the weight distribution of the codes, but also about the geometric structure of their set of supports. It turns out that the complements of supports form the incidence vectors of points and finite subspaces. This result could be helpful in studying the dimension of a design (see \cite{tonchev:1999}) or the weight enumeration of the higher order Reed-Muller codes (see \cite{assmus:1992a}).

\section{Codes and weights}

More details about the theory in this section can be found in Section 5 of \cite{jurrius:2011}.

\subsection{Generalized weight enumerator}

We start with generalizing the weight distribution in the following way, first formulated by Kl{\o}ve \cite{klove:1978,klove:1992} and re-discovered by Wei \cite{wei:1991}. Let $C$ be a linear $[n,k]$ code over $\GF{q}$. Instead of looking at words of $C$, we consider all the subcodes of $C$ of a certain dimension $r$. The \emph{support of a subcode} is the union of the supports of all words in the subcode. The \emph{weight of a subcode} is the size of its support. It is equal to $n$ minus the number of coordinates that are zero for every word in the subcode. The smallest weight for which a subcode of dimension $r$ exists, is called the \emph{$r$-th generalized Hamming weight} of $C$ and denoted by $d_r$. For each $r$ we can define the \emph{$r$-th generalized weight distribution} of the code, that forms the coefficients of the following polynomial.
\begin{definition}
The \emph{generalized weight enumerator} is given by
\[ W^{(r)}_C(X,Y)=\sum_{w=0}^nA^{(r)}_wX^{n-w}Y^w, \]
where $A^{(r)}_w=|\{D\subseteq C:\dim D=r, \wt{D}=w\}|$.
\end{definition}
We can see from this definition that $A^{(0)}_0=1$ and $A^{(r)}_0=0$ for all $0<r\leq k$. Furthermore, every 1-dimensional subspace of $C$ contains $q-1$ non-zero codewords, so $(q-1)A^{(1)}_w$ is the number of words of weight $w$ for $0<w\leq n$.

\subsection{Extended weight enumerator}

Let $C$ be a linear $[n,k]$ code over $\GF{q}$ with generator matrix $G$. Then we can form the $[n,k]$ code $C\otimes\GF{q^m}$ over $\GF{q^m}$ by taking all $\GF{q^m}$-linear combinations of the codewords in $C$. We call this the \emph{extension code} of $C$ over $\GF{q^m}$. By embedding its entries in $\GF{q^m}$, we find that $G$ is also a generator matrix for the extension code $C\otimes\GF{q^m}$. This motivates the usage of $T$ as a variable for $q^m$ in the next definition.
\begin{definition}
The \emph{extended weight enumerator} is the polynomial
\[ W_C(X,Y,T)= \sum_{w=0}^{n} A_w(T) X^{n-w}Y^w \]
where the $A_w(T)$ are integral polynomials in $T$ and $A_w(q^m)$ is the number of codewords of weight $w$ in $C\otimes\GF{q^m}$.
\end{definition}
We will omit the proof that the $A_w(T)$ are indeed polynomials of degree at most $k$.

\subsection{Connections}

There is a connection between the extended and generalized weight enumerator.
\begin{theorem}\label{ewe-gwe}
Let $C$ be a linear $[n,k]$ code over $\GF{q}$. Then the extended weight numerator and the generalized weight enumerators are connected via
\[ W_C(X,Y,T)=\sum_{r=0}^k\left(\prod_{j=0}^{r-1}(T-q^j)\right)W^{(r)}_C(X,Y). \]
\end{theorem}
We need the following proposition to prove the case $T=q^m$.
\begin{proposition}\label{bij}
Let $C$ be a $[n,k]$ code over $\GF{q}$, and let $C^m$ be the linear subspace consisting of the $m\times n$ matrices over $\GF{q}$ whose rows are in $C$. Then there is an isomorphism of $\GF{q}$-vector spaces between $C\otimes\GF{q^m}$ and $C^m$.
\end{proposition}
\begin{proof}
Choose a primitive $m$-th root of unity $\alpha\in\GF{q^m}$. For a word in $C\otimes\GF{q^m}$, write all the coordinates on the basis $(1,\alpha,\alpha^2,\ldots,\alpha^{m-1})$. This gives an $m\times n$ matrix over $\GF{q}$ whose rows are words of $C$. We will omit the details showing this is an isomorphism.
\end{proof}

Note that this isomorphism depends on the choice of a primitive element $\alpha$. The use of this isomorphism for the proof of Theorem \ref{ewe-gwe} was suggested in \cite{simonis:1993} by Simonis.
\begin{corollary}\label{gen-ext-wt}
Let $\mathbf{c}\in C\otimes\GF{q^m}$ and $M\in C^m$ the corresponding $m\times n$ matrix under a given isomorphism. Let $D\subseteq C$ be the subcode generated by $M$. Then $\supp{\mathbf{c}}=\supp{D}$ and hence $\wt{\mathbf{c}}=\wt{D}$.
\end{corollary}
A counting-argument gives the next result.
\begin{proposition}
Let $C$ be a $[n,k]$ code over $\GF{q}$. Then the weight numerator of an extension code and the generalized weight enumerators are connected via
\[ A_w(q^m)=\sum_{r=0}^mA^{(r)}_w \prod_{i=0}^{r-1}(q^m-q^i). \]
\end{proposition}
This result first appears in \cite[Theorem 3.2]{helleseth:1977}. It gives Theorem \ref{ewe-gwe} by Lagrange interpolation.

\subsection{A geometric point of view}

A \emph{projective system} $\mathcal{P}=(P_1,\ldots,P_n)$ in $\PG{r,q}$, the projective space over $\GF{q}$ of dimension $r$, is an enumeration of points $P_j$ in this projective space, such that not all these points lie in a hyperplane. Let $P_j$ be given by be homogeneous coordinates $(p_{0j}:p_{1j}:\cdots :p_{rj})$ and let $G_{\mathcal{P}}$ be the $(r+1)\times n$ matrix with $(p_{0j},p_{1j},\ldots ,p_{rj})^T$ as $j$-th column. Then $G_{\mathcal{P}}$ is the generator matrix of a nondegenerate linear code over $\GF{q}$ of length $n$ and dimension $r+1$, since not all points lie in a hyperplane. \\
Conversely, let $G$ be a generator matrix of a nondegenerate linear $[n,k]$ code $C$ over $\GF{q}$, so $G$ has no zero columns. Take the columns of $G$ as homogeneous coordinates of points in $\PG{s-1,q}$. This gives the projective system $\mathcal{P}_G$ over $\GF{q}$ of $G$. \\

From this two notions it can be shown that there is a one-to-one correspondence between generalized equivalence classes of linear $[n,k]$ codes, and equivalence classes of projective systems with $n$ points in $\PG{k-1,q}$. See \cite{katsman:1987,tsfasman:1991}. \\

Remember we can write a codeword $\mathbf{c}\in C$ as $\mathbf{c}=\mathbf{x}G$, with $\mathbf{x}\in\GF{q}^k$. The $i$-th coordinate of $\mathbf{c}$ is zero if and only if the standard inner product of $\mathbf{x}$ and the $i$-th column of $G$ is zero. So in terms of projective systems, $P_i$ is in the hyperplane perpendicular to $\mathbf{x}$. We can generalize this to subcodes of $C$. \\
Let $\Pi$ be a subspace of co-dimension $r$ in $\PG{k-1,q}$, and let $M$ be an $r\times k$ matrix whose nullspace is $\Pi$. Then $MG$ is an $r\times n$ matrix of full rank whose rows are a basis of a subcode $D\subseteq C$. This gives a one-to-one correspondence between subspaces of codimension $r$ of $\PG{k-1,q}$ and subcodes of $C$ of dimension $r$. This correspondence is independent of the choice of $M$, $G$ and the basis of $D$; see \cite{tsfasman:1995} for details.

\begin{theorem}\label{gen-supp}
Let $D\subseteq C$ be a subcode of dimension $r$, and $\Pi\subseteq\PG{k-1,q}$ the corresponding subspace of codimension $r$. Then a coordinate $i\in[n]$ is in $[n]\setminus\supp{D}$ if and only if the point $P_i\in\mathcal{P}_G$ is in $\Pi$.
\end{theorem}
\begin{proof}
The $i$-th coordinate of $D$ is zero for all words in $D$ if and only if all elements in the basis of $D$ have a zero in the $i$-th coordinate. This happens if and only if the $i$-th column of $G$ is in the nullspace of $M$, or, equivalently, if the point $P_i\in\mathcal{P}_G$ is in $\Pi$.
\end{proof}
\begin{corollary}\label{gen-wt}
Let $D\subseteq C$ be a subcode of dimension $r$, and $\Pi\subseteq\PG{k-1,q}$ the corresponding subspace of codimension $r$. Then the weight of $D$ is equal to $n$ minus the number of points $P_i\in\mathcal{P}_G$ that are in $\Pi$.
\end{corollary}

\section{Codes form a finite projective space}

Consider the projective system $\mathcal{P}$ that consists of all the points in $\PG{s-1,q}$ without multiplicities. The corresponding code is the Simplex code:

\begin{definition}
The $q$-ary Simplex code $\mathcal{S}_q(s)$ is a linear $[(q^s-1)/(q-1),s]$ code over $\GF{q}$. The columns of the generator matrix of the code are all possible nonzero vectors in $\GF{q}^s$, up to multiplication by a scalar.
\end{definition}

The correspondence between $\mathcal{P}$ and the Simplex code is independent of the choice of a generator matrix. We use this correspondence to determine the extended weight enumerator of the Simplex code. We do this via the generalized weight enumerators.

\begin{theorem}
The generalized weight enumerators of the Simplex code $\mathcal{S}_q(s)$ are, for $0\leq r\leq s$, given by
\[ W_{\mathcal{S}_q(s)}(X,Y)=\left[{s\atop r}\right]_qX^{(q^{s-r}-1)/(q-1)}Y^{(q^s-q^{s-r})/(q-1)} . \]
\end{theorem}
\begin{proof}
We use Corollary \ref{gen-wt} to determine the weights of all subcodes of $\mathcal{S}_q(s)$. Fix a dimension $r$. Let $D\subseteq\mathcal{S}_q(s)$ be some subcode of dimension $r$, that corresponds to the subspace $\Pi\subseteq\PG{s-1,q}$ of codimension $r$. The weight of $D$ is equal to $n$ minus the number of points in $\mathcal{P}$ that are in $\Pi$. Because all points of $\PG{s-1,q}$ are in $\mathcal{P}$, the weight is the same for all $D$ and it is equal to $n$ minus the total number of points in $\Pi$. This means the weight of $D$ is equal to
\[ \frac{q^s-1}{q-1}-\frac{q^{s-r}-1}{q-1}=\frac{q^s-q^{s-r}}{q-1} \]
and the theorem follows.
\end{proof}

From the previous calculation and Theorem \ref{gen-supp} the next statement follows.

\begin{corollary}
Let $D$ be some subcode of dimension $r$ of the Simplex code $\mathcal{S}_q(s)$. Then the points in $\mathcal{P}$ indexed by $[n]\setminus\supp{D}$ are all the points in the corresponding subspace $\Pi$ of codimension $r$ in $\PG{s-1,q}$.
\end{corollary}

We can now write down the extended weight enumerator of the Simplex code:

\begin{theorem}
The extended weight enumerator of the Simplex code $\mathcal{S}_q(s)$ is equal to
\[ W_{\mathcal{S}_q(s)}(X,Y,T)=\sum_{r=0}^s\left(\prod_{j=0}^{r-1}(T-q^j)\right)\left[{s\atop r}\right]_qX^{(q^{s-r}-1)/(q-1)}Y^{(q^s-q^{s-r})/(q-1)} . \]
\end{theorem}
\begin{proof}
We use the correspondence between the generalized and extended weight enumerator in Theorem \ref{ewe-gwe}:
\begin{align*}
W_{\mathcal{S}_q(s)}(X,Y,T) & = \sum_{r=0}^s\left(\prod_{j=0}^{r-1}(T-q^j)\right)W_{\mathcal{S}_q(s)}(X,Y) \\
 & = \sum_{r=0}^s\left(\prod_{j=0}^{r-1}(T-q^j)\right)\left[{s\atop r}\right]_qX^{(q^{s-r}-1)/(q-1)}Y^{(q^s-q^{s-r})/(q-1)}. \\
 & \qedhere
\end{align*}
\end{proof}

In combination with the isomorphism of Proposition \ref{bij} and Corollary \ref{gen-ext-wt}, we get the following consequence.

\begin{corollary}\label{gwe-design}
The points in $\mathcal{P}$ indexed by the complement of the support of a word of weight $(q^s-q^{s-m})/(q-1)$ in the extension code $\mathcal{S}_q(s)\otimes\GF{q^m}$ are all the points in a subspace of $\PG{s-1,q}$ of codimension $m$.
\end{corollary}

\begin{example}\label{simplex23}
We consider the Simplex code $\mathcal{S}_2(3)$. It is a binary $[7,3]$ code. Its extended weight enumerator has coefficients
\begin{align*}
A_0(T) & = 1 \\
A_1(T) & = 0 \\
A_2(T) & = 0 \\
A_3(T) & = 0 \\
A_4(T) & = 7(T-1) \\
A_5(T) & = 0 \\
A_6(T) & = 7(T-1)(T-2) \\
A_7(T) & = (T-1)(T-2)(T-4).
\end{align*}
Note that for any code we have $A_0(T)=1$ for the zero word, and all other polynomials are divisible by $(T-1)$ because over a field of size 1 we have only the zero word. In the binary case $T=2$, the polynomials for $A_6(T)$ and $A_7(T)$ vanish and the code has only one nonzero weight. For $T=2^2=4$ still $A_7(T)$ vanishes, it is a two-weight code. For $T=2^3$ and higher extensions we get all the three possible nonzero weights. 
\end{example}

\section{Codes form a finite affine space}

It might sound a bit strange to talk about the projective system coming from an affine space. To solve this, remember that we can construct the finite affine space $\AG{s-1,q}$ by deleting a hyperplane from $\PG{s-1,q}$. So let the projective system $\mathcal{P}$ consists of all points in $\PG{s-1,q}$ minus the points in a hyperplane $H$ of $\PG{s-1,q}$.  Without loss of generality, we can choose $H$ to be the hyperplane $X_1=0$. The corresponding code is (equivalent to) the first order $q$-ary Reed-Muller code, and we can define it in the following way:

\begin{definition}
The first order $q$-ary Reed-Muller code $\mathcal{RM}_q(1,s-1)$ is a linear $[q^{s-1},s]$ code over $\GF{q}$. The generator matrix consists of the all-one row, and the other positions in the columns of the generator matrix are all possible vectors in $\GF{q}^{s-1}$.
\end{definition}

Note that the linear dependence between the columns of the generator matrix is now equal to the dependence between the corresponding affine points: this property is very useful if we want to talk about the matroid associated to the code, see \cite{mphako:2000}. \\

We will use the projective system described above to determine the extended weight enumerator of the first order Reed-Muller code. We do this via the generalized weight enumerators.

\begin{theorem}
The generalized weight enumerators of the first order Reed-Muller code $\mathcal{RM}_q(1,s-1)$ are, for $0<r<s$, given by
\[ W_{\mathcal{RM}_q(1,s-1)}^{(r)}(X,Y)=\left[{s-1\atop r-1}\right]_qY^n+q^r\left[{s-1\atop r}\right]_qX^{q^{s-1-r}}Y^{q^{s-1}-q^{s-1-r}} . \]
The extremal cases are, as always, given by 
\begin{eqnarray*}
W_{\mathcal{RM}_q(1,s-1)}^{(0)}(X,Y) & = & X^n, \\
W_{\mathcal{RM}_q(1,s-1)}^{(s)}(X,Y) & = & Y^n.
\end{eqnarray*}
\end{theorem}
\begin{proof}
We use Corollary \ref{gen-wt} to determine the weights of all subcodes of $\mathcal{RM}_q(1,s-1)$. Fix a dimension $r$, with $0\leq r\leq s$. Let $D\subseteq\mathcal{RM}_q(1,s-1)$ be some subcode of dimension $r$, that corresponds to the subspace $\Pi\subseteq\PG{s-1,q}$ of codimension $r$. The weight of $D$ is equal to $n$ minus the number of points in $\mathcal{P}$ that are in $\Pi$. There are two possibilities:
\begin{itemize}
\item{$\Pi\subseteq H$;}
\item{$\Pi\not\subseteq H$.}
\end{itemize}
In the first case, we cannot have $r=0$, since then $\Pi$ is the whole of $\PG{s-1,q}$ and this cannot be contained in the hyperplane $H$. So let $r>0$. Now non of the points of $\mathcal{P}$ are in $\Pi$, since no points of $H$ are in $\mathcal{P}$. So $\supp{D}=[n]$ and $\wt{D}=n$. The number of such codes is equal to the number of subspaces of codimension $r-1$ in $H\cong PG(s-2,q)$, this is $\left[s-1\atop r-1\right]_q$. So for $0<r\leq s$ we get the following term for the generalized weight enumerator:
\[ \left[s-1\atop r-1\right]_qY^n. \]
In the second case, we do not have to consider $r=s$, since then $\Pi$ is the empty set and this was already included in the previous case. So let $r<s$. Now $\Pi$ and $H$ intersect in a subspace of codimension $r$ in $H$. The points of $\mathcal{P}$ that are in $\Pi$, are all those points of $\Pi$ that are not in $\Pi\cap H$. By the construction of the affine space $\AG{s-1,q}$, the points of $\Pi\setminus(\Pi\cap H)$ form a subspace of $\AG{s-1,q}$ of codimension $r$. The number of points in such a subspace is $q^{s-1-r}$, so $\wt{D}=n-q^{s-1-r}=q^{s-1}-q^{s-1-r}$. The number of such codes is equal to the number of subspaces of codimension $r$ in $\AG{s-1,q}$, this is $q^r\left[s-1\atop r\right]_q$. So this case gives the following term for the generalized weight enumerator, for $0\leq r<s$:
\[ q^r\left[{s-1\atop r}\right]_qX^{q^{s-1-r}}Y^{q^{s-1}-q^{s-1-r}}. \]
Summing up this two cases leads to the given formulas.
\end{proof}

From the previous calculation and Theorem \ref{gen-supp} the next statement follows.

\begin{corollary}
Let $D$ be some subcode of dimension $r$ of the first order Reed-Muller code $\mathcal{RM}_q(1,s-1)$. Then either $\supp{D}=[n]$, or the points in $\mathcal{P}$ indexed by $[n]\setminus\supp{D}$ are all the points in the corresponding subspace $\Pi$ of codimension $r$ in $\PG{s-1,q}$.
\end{corollary}

We can now write down the extended weight enumerator of the first order Reed-Muller code:

\begin{theorem}
The extended weight enumerator of the first order Reed-Muller code $\mathcal{RM}_q(1,s-1)$ is equal to
\begin{multline*}
W_{\mathcal{RM}_q(1,s-1)}(X,Y,T)=\sum_{r=1}^s\left(\prod_{j=0}^{r-1}(T-q^j)\right)\left[{s-1\atop r-1}\right]_qY^n \\
+\sum_{r=0}^{s-1}\left(\prod_{j=0}^{r-1}(T-q^j)\right)q^r\left[{s-1\atop r}\right]_qX^{q^{s-1-r}}Y^{q^{s-1}-q^{s-1-r}}.
\end{multline*}
\end{theorem}
\begin{proof}
We use the correspondence between the generalized and extended weight enumerator in Theorem \ref{ewe-gwe}:
\begin{align*}
W_{\mathcal{RM}_q(1,s-1)}(X,Y,T) & = \sum_{r=0}^s\left(\prod_{j=0}^{r-1}(T-q^j)\right)W_{\mathcal{S}_q(s)}(X,Y) \\
 & = X^n+\sum_{r=1}^{s-1}\left(\prod_{j=0}^{r-1}(T-q^j)\right)\left(\left[{s-1\atop r-1}\right]_qY^n\right. \\
  & \phantom{= }\left.+q^r\left[{s-1\atop r}\right]_qX^{q^{s-1-r}}Y^{q^{s-1}-q^{s-1-r}}\right)+Y^n \\
 & = \sum_{r=0}^{s-1}\left(\prod_{j=0}^{r-1}(T-q^j)\right)q^r\left[{s-1\atop r}\right]_qX^{q^{s-1-r}}Y^{q^{s-1}-q^{s-1-r}} \\
 & \phantom{= }+\sum_{r=1}^s\left(\prod_{j=0}^{r-1}(T-q^j)\right)\left[{s-1\atop r-1}\right]_qY^n. \\
 & \qedhere
\end{align*}
\end{proof}

In combination with the isomorphism of Proposition \ref{bij} and Corollary \ref{gen-ext-wt}, we get the following consequence.

\begin{corollary}\label{ewe-design}
The points in $\mathcal{P}$ indexed by the complement of the support of a word of weight $q^{s-1}-q^{s-1-m}$ in the extension code $\mathcal{RM}_q(1,s-1)\otimes\GF{q^m}$ are all the points in a subspace of $\AG{s-1,q}$ of codimension $m$.
\end{corollary}

\begin{example}
We consider the Reed-Muller code $\mathcal{RM}_2(1,3)$. It is a binary $[8,4]$ code. Its extended weight enumerator has coefficients
\begin{align*}
A_0(T) & = 1 \\
A_1(T) & = 0 \\
A_2(T) & = 0 \\
A_3(T) & = 0 \\
A_4(T) & = 14(T-1) \\
A_5(T) & = 0 \\
A_6(T) & = 28(T-1)(T-2) \\
A_7(T) & = 8(T-1)(T-2)(T-4) \\
A_8(T) & = (T-1)(T^3-7T^2+21T-21).
\end{align*}
As noted in Example \ref{simplex23}, for any code we have $A_0(T)=1$ for the zero word, and all other polynomials are divisible by $(T-1)$ because over a field of size 1 we have only the zero word. In the binary case $T=2$, the polynomials for $A_6(T)$ and $A_7(T)$ vanish and we get a two-weight code. For $T=2^2=4$ still $A_7(T)$ vanishes. For $T=2^3$ and higher extensions we get all the four possible nonzero weights. \\
This example and the previous Example \ref{simplex23} also illustrate that the binary Simplex code $\mathcal{S}_2(s)$ is equivalent to the binary Reed-Muller code $\mathcal{RM}_2(1,s)$ shortened at the first coordinate.
\end{example}

\section{Links to other problems}

We found direct formulas for the extended weight enumerator of the $q$-ary Simplex code and the $q$-ary first order Reed-Muller code. Following from this calculations, we found the geometrical structure of the supports of the subcodes and of words in extension codes. This triggers a lot of links with other problems in discrete mathematics and coding theory. The following list is by no means exhaustive, but it hopefully serves as encouragement and inspiration for further research. \\

Mphako \cite{mphako:2000} calculated the Tutte polynomial of the matroids coming from finite projective and affine spaces. He does this by using the equivalence between the Tutte polynomial and the coboundary polynomial. The coboundary polynomial of a matroid is also the reciprocal polynomial of the extended weight enumerator of the code associated to a matroid, see \cite{jurrius:2011} for details. The formulas found by Mphako indeed coincide with the extended weight enumerators in this paper. \\

We calculated the extended weight enumerator for the first order (generalized) Reed-Muller code. The weight enumeration of higher order Reed-Muller codes is an open problem. The generalized Hamming weights of $q$-ary Reed-Muller codes were found by Heijnen and Pellikaan \cite{heijnen:1998}. \\

It is known that the $r$-th order Reed-Muller codes $\mathcal{RM}(r,m)$ arise from the design of points and subspaces of codimension $r$ in the affine space $\AG{m,2}$, see \cite{assmus:1992b}. The $q$-ary analogue of this statement is treated in \cite{assmus:1992a}. In Corollary \ref{ewe-design} we saw that the supports of the words in the extension code $\mathcal{RM}_q(1,s-1)\otimes\GF{q^r}$ contain the design of points and subspaces of codimension $r$ in $\AG{s-1,q}$. This suggests some kind of link between extension codes of the first order Reed-Muller code and the higher order Reed-Muller codes. If we can make this link explicit, it might lead to more insights to the weight enumeration of higher order Reed-Muller codes. \\

We encountered two types of two-weight codes in this paper: the first order Reed-Muller code, and the extension of the Simplex code $\mathcal{S}_q(s)\otimes\GF{q^2}$. How do these codes fit into the classification of two-weight codes from Calderbank and Kantor \cite{calderbank:1986}? \\

For every design, one can talk about its $p$-rank. Tonchev \cite{tonchev:1999} generalized this concept to the \emph{dimension} of a design, and formulated an analogue of the Hamada conjecture. Knowing the extended weight enumerator of the Simplex code could be of help in proving this conjecture for geometric designs.

\section*{Acknowledgment}

The author would like to thank Vladimir Tonchev for coming up with the question about the weight enumerator of the extension codes of the Simplex code, and for his encouraging conversations on the subject.

\bibliographystyle{abbrv}
\bibliography{wtenum}

\end{document}